\documentclass{amsart}

\usepackage{amsmath,amsthm,enumerate,verbatim,graphicx,wrapfig,longtable,mathrsfs}
\usepackage[all,arc]{xy}

\usepackage{hyperref}
\usepackage{tocvsec2}
\usepackage{amssymb}

\usepackage[T2A,T1]{fontenc}

\newcommand{\defof}[1]{\textbf{#1}}

\def\C{\ensuremath{\mathbb{C}}}
\def\N{\ensuremath{\mathbb{N}}}
\def\P{\ensuremath{\mathbb{P}}}
\def\R{\ensuremath{\mathbb{R}}}
\def\Z{\ensuremath{\mathbb{Z}}}

\def\SS{\mathfrak{S}}
\def\TT{\mathfrak{T}}
\def\AA{\mathfrak{A}}
\def\II{\mathfrak{I}}

\DeclareMathOperator\codim{codim}
\DeclareMathOperator\ec{ec}
\DeclareMathOperator\rk{rk}
\DeclareMathOperator\mspan{span}

\def\Mat{\ensuremath{\mathrm{Mat}}}

\hypersetup{
    colorlinks,%
    citecolor=black,%
    filecolor=black,%
    linkcolor=black,%
    urlcolor=black
}

\theoremstyle{plain}
\newtheorem{thm}{Theorem}[section]
\newtheorem{prop}[thm]{Proposition}
\newtheorem{lem}[thm]{Lemma}
\newtheorem{cor}[thm]{Corollary}

\theoremstyle{definition}
\newtheorem{cons}[thm]{Construction}
\newtheorem{defn}[thm]{Definition}
\newtheorem{defns}[thm]{Definitions}
\newtheorem{exmp}[thm]{Example}

\newtheorem{cex}[thm]{Counterexample}

\numberwithin{figure}{section}
\numberwithin{table}{section}

\title{The Expected Codimension of a Matroid Variety}
\author{Nicolas Ford}

\begin{document}

\maketitle

\section*{Abstract}
Matroid varieties are the closures in the Grassmannian of sets of points defined by specifying which Pl\"ucker coordinates vanish and which don't. In general these varieties are very ill-behaved, but in many cases one can estimate their codimension by keeping careful track of the conditions imposed by the vanishing of each Pl\"ucker coordinates on the columns of the matrix representing a point of the Grassmannian. This paper presents a way to make this procedure precise, producing a number for each matroid variety called its expected codimension that can be computed combinatorially solely from the list of Pl\"ucker coordinates that are prescribed to vanish. We prove that for a special, well-studied class of matroid varieties called positroid varieties, the expected codimension coincides with the actual codimension.

\tableofcontents

\settocdepth{section}

\section{Introduction}
Consider a point $x$ on the Grassmannian $G(k,n)$ of $k$-planes in $\C^n$. The \defof{matroid} of $x$ is defined to be the set of Pl\"ucker coordinates that are nonzero at $x$, and a \defof{matroid variety} is the closure of the set of points on $G(k,n)$ with a particular matroid. Many enumerative problems on the Grassmannian can be described in terms of matroid varieties; the Schubert varieties that form the usual basis for the cohomology ring of $G(k,n)$ are an especially well-behaved special case.

Unfortunately, matroid varieties can be very ugly in full generality. A good start toward understanding the behavior of a matroid variety would be to find some way to compute its dimension directly from the matroid that defines it, but even this has very little hope of succeeding.

Still, one can come up with an estimate of the codimension of a matroid variety inside its Grassmannian by keeping careful track of the conditions imposed by the vanishing of Pl\"ucker coordinates on the columns of the $k\times n$ matrix defining a point on $G(k,n)$. This paper is about a way to make this idea precise, producing a number called the \defof{expected codimension} for each matroid. While it doesn't always produce the actual codimension of the matroid variety, we will prove that it always does for \defof{positroids}, a particularly well-studied class of matroids which includes both Schuberts and Richardsons.

We write $[n]$ for the set $\{1,2,\ldots,n\}$, and for any set $S$ we write $\binom Sk$ for the set of all $k$-element subsets of $S$. $G(k,n)$ will always stand for the Grassmannian of $k$-planes in $\C^n$. For $S\in\binom{[n]}k$, we write $p_S$ for corresponding Pl\"ucker coordinate on $G(k,n)$; that is, thinking of elements of $G(k,n)$ as being represented by $k\times n$ matrices, $p_S$ is the determinant of the minor whose columns are the elements of $S$. All varieties in this paper are over $\C$.

\begin{figure}[t]
\includegraphics[width=6cm]{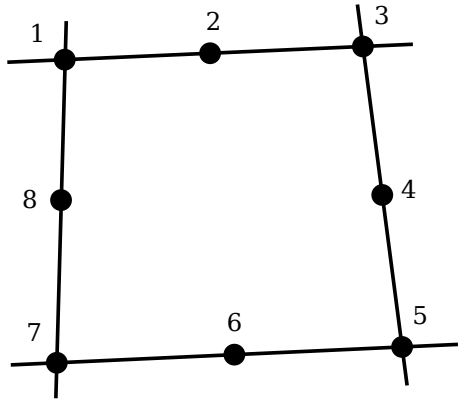}
\caption{A projective model of the ``square matroid.''}
\label{fig:squarepos}
\end{figure}

The procedure we follow was described quite informally in \cite[3.3]{fnrmat}, and we flesh it out here. We can draw a picture to represent a point in $G(k,n)$ by placing $n$ points in $\C^k$ or $\P^{k-1}$, each representing the corresponding column; we'll call these pictures \defof{projective models}. Consider the matroid $S$ of a point in $G(3,8)$ for which $p_{123}$, $p_{345}$, $p_{567}$ and $p_{178}$ are the only Pl\"ucker variables that vanish. A projective model for $S$ using points in $\P^2$ is shown in Figure \ref{fig:squarepos}.

We estimate the codimension of $X(S)$ in $G(3,8)$ as follows. To build a projective model of $S$ like the one in the figure, we are free to place the odd-numbered points wherever we want. Once we've done this, each even-numbered point is forced to live in the codimension-1 subspace spanned by two of the points we've already placed. So we guess that the codimension of $X(S)$ is $1+1+1+1=4$.

This turns out to be the correct answer for $\codim X(S)$, and we'll see later that the reasoning given is more or less why. One immediate question is whether the result of this procedure depends on the order in which we ``place'' the points. Once we've nailed down exactly what the procedure is, we will see that the answer to this question is no. For now, let's just try a couple more. If they're placed in order starting from the beginning, points 1, 2, 4, 5, and 7 can be put anywhere without restriction. As before, points 6 and 8 are now forced onto codimension-1 subspaces. But point 3 is now forced onto a codimension-2 subspace: it needs to be on the intersection of $\mspan\{1,2\}$ and $\mspan\{4,5\}$. So, adding all the restrictions up, we get $1+1+2=4$. Similarly, we could get ``$2+2$'' by placing the points 1, 2, 4, 5, 6, and 8 freely, and then putting 3 and 5 in last.

We will show that our definition is independent of the order by recasting it in terms of something manifestly order-independent. In $G(k,n)$, specifying exactly which Pl\"ucker coordinates vanish is the same as describing, for each subset of the set of columns, the dimension of its span in $\C^k$; in matroid language, this is called the \defof{rank} of the corresponding subset $[n]$. Our current procedure is to ask, for each element, what constraints are put on that element when it's added in. Instead let's ask, for each subset of the base set of the matroid, what constraints it puts on its elements. For example, in the set $\{1,2,3\}$ in $S$, the third element added in will be forced onto a codimension-1 subspace no matter what the order is; the only thing that matters is that the number of elements of this set is 1 more than its rank.

So it seems like we should add up the numbers $(\#F-\rk F)(k-\rk F)$ for each subset $F$; the first factor is the number of elements which are constrained by $F$ and the second is the codimension of the subspace those elements are constrained to. But this is not quite right: whenever an element belongs to two different such $F$'s, it's going to be counted twice. Sometimes this is desirable, as we saw with point 3 two paragraphs up, but often it will be redundant, as it is for the sets $\{1,2,3\}$ and $\{1,2,3,4\}$ in $S$. We ought to subtract 1 from the number of constrained elements for the larger set to account for the fact that it was already taken care of by the smaller one.

This, finally, takes us to the definition that we'll be using:

\begin{defn}
\label{def:expdim}
Let $M$ be a matroid of a point of $G(k,n)$, and let $\SS$ be a collection of subsets of $[n]$. For $S\in\SS$, we define \[c(S)=\#S-\rk S,\] and \[a_{\SS}(S)=c(S)-\sum_{\substack{T\in\SS\\T\subsetneq S}}a_{\SS}(T),\quad a_\SS(\varnothing)=0,\] where the sum goes over elements of $\SS$. (Note that this indeed recursively defines $a$ for all elements of $\SS$.) We then define the \defof{expected codimension of $M$ with respect to $\SS$} to be \[\ec_{\SS}(M):=\sum_{S\in\SS}(k-\rk S)a_{\SS}(S).\] The \defof{expected codimension of $M$} is then \[\ec(M):=\ec_{\mathcal{P}(E)}(M).\] (Similarly, we will write $a=a_{\mathcal{P}([n])}$.) We say that $M$ \defof{has expected codimension} if $\ec(M)$ is equal to the codimension of $X(M)$ in $G(k,n)$.
\end{defn}

Allowing $\SS$ to be something other than $\mathcal{P}([n])$ itself might seem strange, but it will turn out to be very helpful. We will show that in many cases $\ec_{\SS}$ will be the same for many different choices of $\SS$ but easier to compute for some choices than for others, and we will be happy to have the flexibility, for both theoretical and practical reasons.

In Definition \ref{def:positroid}, we describe an important class of matroids called \defof{positroids}. We will show in Theorem \ref{thm:posexpcodim} that positroids have expected codimension. In Section \ref{sec:valuativity} we also discuss \defof{valuativity}, a well-studied property of some numerical invariants of matroids, and show that expected codimension is valuative.

\subsection*{Acknowledgements}

I am incredibly grateful to my advisor, David Speyer, for many fruitful conversations and ideas, some of which appear in this paper, and to Allen Knutson for the same. I also want to thank Jordan Watkins both for talking to me about the content of this paper and for his meticulous proofreading. This work was partially supported by NSF grant DMS 0943832.

\section{Matroids and Matroid Varieties}
\subsection{Matroids}

We will need to have access to some theoretical results about abstract matroids. There are many equivalent definitions of matroids, all useful in different contexts, and we are only going to mention two of them here. A good place to learn more about matroids from a combinatorial perspective is \cite{whitemat}.

Given a collection of vectors in a vector space, its matroid combinatorially captures all the information about the linear relations among these vectors. We will consider two equivalent ways to do this. Details of these and other axiomatizations of matroids can be found in \cite[pp. 298--312]{whitemat}

\begin{defn}
\label{def:basismat}
A matroid may be specified in terms of its \defof{bases}. According to this definition, a matroid $M$ is a finite set $E$ together with a collection of subsets $\mathscr{B}\subseteq\mathcal{P}(E)$. (An element of $\mathscr{B}$ is called a basis.) We require:
\begin{itemize}
\item $\mathscr{B}$ is not empty.
\item No element of $\mathscr{B}$ contains another.
\item For $B,B'\in\mathscr{B}$ and $x\in B$, there is some $y\in B'$ so that $B-\{x\}\cup\{y\}\in\mathscr{B}$.
\end{itemize}
Note that this is enough to force all bases to have the same number of elements.
\end{defn}

\begin{defn}
\label{def:matroidvector}
Suppose we have a finite-dimensional vector space $V$, a finite set $E$, and a function $e:E\to V$ whose image spans $V$. We can put a matroid structure on $E$ by taking $\mathscr{B}$ to be the collection of all subsets of $E$ which map \emph{injectively} to a basis of $V$. (The reason for this funny definition is that we'd like to be able to take the same element of $V$ more than once; otherwise $E$ could just be a subset of $V$. We will hardly ever be careful about the difference between an element $x\in E$ and its image $e(x)\in V$.) It's an easy exercise to show that this definition satisfies the axioms above. Matroids which arise in this way are called \defof{realizable}.
\end{defn}

The following terminology will be useful. Most of these definitions mirror the corresponding ones from linear algebra in the realizable case.

\begin{defns}
\label{def:matroiddefs}
Let $M$ be a matroid on a set $E$.
\begin{enumerate}
\item A subset of $E$ which is contained in a basis is called \defof{independent}. Any other set is \defof{dependent}.
\item For $F\subseteq E$, the \defof{rank} of $F$, written $\rk F$, is the size of the largest independent set contained in $F$. Note that $\rk E$ is the same as the size of any basis. We define $\rk M$ to be $\rk E$.
\item For a set $F$ and an element $x\in E$, we say that $x$ is in the \defof{closure} of $F$, written $x\in\overline F$, if $\rk(F\cup\{x\})=\rk F$. Note that, as the name suggests, closure is idempotent and inclusion-preserving. Sets which are their own closures are called \defof{flats}. In the realizable case, the flats are the intersections of subspaces of $V$ with $E$.
\item A set $F$ which contains a basis is called a \defof{spanning set}. Equivalently, $F$ spans if $\rk F=\rk E$, or if $\overline F=E$.
\item If $\rk\{x\}=0$, we say $x$ is a \defof{loop}. Equivalently, $x$ is not in any basis, or $x\in\overline\varnothing$, or $x$ is in every flat. In the realizable case, loops are elements of $E$ which map to the zero vector in $V$.
\item If $\rk(E-\{x\})=\rk E-1$, we say $x$ is a \defof{coloop}. Equivalently, $x$ is in every basis.
\item If $\rk\{x,y\}=1$, we say that $x$ and $y$ are \defof{parallel}. Equivalently, any flat which contains one of $x$ or $y$ also contains the other.
\end{enumerate}
\end{defns}

It will also be convenient to note that matroids can be defined just by listing the axioms that have to be satisfied by the rank function defined above:

\begin{defn}
\label{def:rankmat}
A matroid may be specified in terms of the ranks of all its subsets. According to this definition, a matroid is a finite set $E$ together with a function $\rk:\mathcal{P}(E)\to\N$ satisfying:
\begin{itemize}
\item $\rk\varnothing=0$.
\item $\rk(F\cup\{x\})$ is either $\rk F$ or $\rk F+1$.
\item If $\rk F=\rk(F\cup\{x\})=\rk(F\cup\{y\})$, then $\rk(F\cup\{x,y\})=\rk F$.
\end{itemize}
Note that this is enough to force the useful inequality $\rk A+\rk B\ge\rk(A\cup B)+\rk(A\cap B)$.
\end{defn}

Given the same data we used to define a realizable matroid before --- a set $E$ with a function $e$ to a vector space $V$ --- we can get a rank function on $E$ by setting $\rk(F)=\dim(\mspan(e(F)))$.

We have already mentioned how to turn a collection of bases into a rank function. To go the other way, we can say $B$ is a basis if it is minimal among sets of maximal rank. One can check that these two correspondences make the two definitions given here equivalent. We will not distinguish between them as we go forward.

\subsection{Matroid Varieties}

As mentioned above, the main objects of study in this paper are certain subvarieties of Grassmannians which can be described in terms of matroids.

\begin{cons}
\label{cons:pointmat}
Consider the Grassmannian $G(k,n)$, which we'll think of as the set of $k\times n$ matrices of full rank modulo the obvious left action of $GL_k$. When one builds the Grassmannian in this way, one ordinarily considers the $k$ rows of the matrix as elements of $\C^n$, and the action of $GL_k$ corresponds to automorphisms of the span of those elements, so that we are left with a variety that parametrizes the $k$-planes in $\C^n$.

We will think about our matrices the other way. Given a $k\times n$ matrix of full rank, consider the function $e:[n]\to\C^k$ which takes $i$ to the $i$'th column of our matrix. We can then use Definition \ref{def:matroidvector} to put a matroid structure on $[n]$. Since the action of $GL_k$ clearly doesn't change which matroid we get, we have assigned a matroid in a consistent way to every point of the Grassmannian. The Pl\"ucker coordinate $p_S$ corresponding to some $S\in\binom{[n]}k$ is given by the determinant of the submatrix defined by taking the columns in $S$. So $p_S$ vanishes precisely when these $k$ columns fail to span $\C^k$, that is, precisely when $S$ fails to be a basis of our matroid.

Given a matroid $M$ of rank $k$ on $[n]$, the \defof{open matroid variety} $X^\circ(M)$ is the subset of $G(k,n)$ consisting of all points whose matroid is $M$. This is a locally closed subvariety of $G(k,n)$: it is defined by taking all Pl\"ucker coordinates corresponding to bases of $M$ to be nonzero and all the other Pl\"ucker coordinates to be zero. The closure of $X^\circ(M)$ is called the \defof{matroid variety} $X(M)$. Similarly, we can define a matroid variety inside $\Mat_{k\times n}$ in the same way. The open matroid variety in $\Mat_{k\times n}$ doesn't intersect the subvariety of matrices of less than full rank, but its closure will in general.
\end{cons}

The reader who is familiar with the definition of Schubert varieties may be tempted to ignore the definition above and take $X(M)$ to be the subvariety of $G(k,n)$ defined by setting all the Pl\"ucker coordinates corresponding to nonbases of $M$ to zero. Sadly, this is not the same: 

\begin{cex}
\label{cex:pluckergen}
Consider the rank-3 matroid $A$ on $[7]$ generated by the conditions that $\{1,2,7\}$, $\{3,4,7\}$, and $\{5,6,7\}$ have rank 2. The variety $X(A)$ is not cut out by the ideal $(p_{127},p_{347},p_{567})$. That ideal cuts out two components: $X(A)$ and the variety of the matroid in which 7 is a loop. The ideal of $X(A)$ is actually $(p_{127},p_{347},p_{567},p_{124}p_{356}-p_{123}p_{456})$.
\end{cex}

\subsection{Operations on Matroids and Matroid Varieties}

In general, as mentioned in the introduction, matroid varieties are under no obligation to be geometrically well-behaved. They don't have to be irreducible, equidimensional, normal, or even generically reduced (if given the appropriate scheme structure), and even the problem of determining whether $X(M)$ is empty or not is NP-hard (\cite{stretch}). (See \cite{vakilmurphy} for a discussion of how bad these varieties can get.) Still, our goal in this paper is to find some way to control the codimension of a matroid variety, at least in some nice cases. With this in mind, we establish some results which describe the effects of some simple matroid operations on the corresponding matroid varieties.

\begin{defns}
\label{def:directsum}
~\begin{enumerate}
\item Let $M$ be a matroid on $E$ and $N$ a matroid on $F$. The \defof{direct sum of $M$ and $N$} is the matroid $M\oplus N$ on $E\sqcup F$ defined by \[\rk_{M\oplus N}(S)=\rk_M(S\cap E)+\rk_N(S\cap F).\]
\item If $M$ is a matroid, the \defof{loop extension} of $M$ is the matroid formed by taking the direct sum of $M$ with the unique matroid of rank 0 on the one-element set $\{x\}$, so that the new element $x$ is a loop.
\item The \defof{coloop extension} of $M$ is the matroid formed by taking the direct sum of $M$ with the unique matroid of rank 1 on $\{x\}$, so that $x$ is a coloop.
\end{enumerate}
\end{defns}

It's straightforward to compute the codimension of $X(M\oplus N)$ given the codimensions of $X(M)$ and $X(N)$.

\begin{prop}
\label{prop:dimdirectsum}
If $X(M)\subseteq G(k,n)$ and $X(N)\subseteq G(k',n')$ are matroid varieties, then $\codim X(M\oplus N)=\codim X(M)+\codim X(N)$.
\end{prop}

\begin{cor}
\label{cor:loopcoloop}
Let $X(M)\subseteq G(k,n)$ be a matroid variety. Then $\dim X(M\oplus x_0)=\dim X(M\oplus x_1)=\dim X(M)$.
\end{cor}

\begin{defn}
\label{def:connected}
If $M$ can't be written as a direct sum in a nontrivial way, we say that $M$ is \defof{connected}. If we write $M=\bigoplus_iA_i$ with each $A_i$ connected, then the $A_i$'s are uniquely determined, and we call them the \defof{connected components} of $M$.
\end{defn}

There are two other, equivalent ways to define connectedness:
\begin{itemize}
\item $M$ is connected if there is no proper, nonempty subset $S\subseteq E$ for which $\rk S+\rk(E-S)=\rk E$.
\item A \defof{circuit} of $M$ is a minimal dependent set, that is, a dependent set $C$ for which every proper subset is independent. We can define an equivalence relation on $E$ by saying $x$ is equivalent to $y$ if either $x=y$ or there is a circuit of $M$ containing both $x$ and $y$. The connected components of $M$ are the equivalence classes under this relation.
\end{itemize}

\begin{defn}
\label{def:dual}
Let $M$ be a matroid of rank $k$ on a set $E$. The \defof{dual} of $M$ is the matroid $M^*$ on $E$ whose bases are exactly the complements of bases of $M$.
\end{defn}

The rank of a set $S$ in $M^*$ works out to be $\#S-k+\rk(E\setminus S)$. In particular $M^*$ has rank $\#E-k$. The following result follows directly from Definition \ref{def:dual} and Construction \ref{cons:pointmat}.

\begin{prop}
\label{prop:dualiso}
There is an isomorphism $\omega:G(k,n)\to G(n-k,n)$ which takes $p_S$ to $p_{[n]-S}$ for $S\in\binom{[n]}{k}$. For a rank-$k$ matroid $M$ on $[n]$ this restricts to an isomorphism $X(M)\cong X(M^*)$.
\end{prop}

It's straightforward to check that $M$ is connected if and only if $M^*$ is. In fact, $(A\oplus B)^*=A^*\oplus B^*$.

Finally, for a matroid $M$ on a set $E$, we define two different ways to put the structure of a matroid on a subset of $E$. One of them corresponds to restricting to a subspace of a vector space, and the other corresponds to taking a quotient of vector spaces.

\begin{defn}
\label{def:restrictcontract}
Suppose $S\subseteq E$. The \defof{restriction of $M$ to $S$} is the matroid $M|_S$ on $S$ in which the rank of any subset of $S$ is just its rank in $M$. In particular, $\rk(M|_S)=\rk S$. We'll sometimes also refer to this matroid as the result of \defof{deleting} $E-S$. In this context, when we refer to $S$ itself as a matroid, we will always mean the restriction to $S$.

The \defof{contraction of $S$} is the matroid $M/S$ on $E-S$ in which the rank of any set $T$ is $\rk_M(T\cup S)-\rk_M(S)$. In particular, $\rk(M/S)=\rk_ME-\rk_MS$.
\end{defn}

It is important to note that these two constructions are dual to each other. That is, $(M|_S)^*=M^*/(E-S)$, and $(M/S)^*=M^*|_{E-S}$.

\section{Properties of the Expected Codimension}
We now study how $\ec_\SS$ chages for different choices of $\SS$. Throughout this section, $M$ is a matroid of rank $k$ on a set $E$.

First, it will be helpful to write $\ec$ in a more symmetrical way. Thinking of $\SS$ as a poset under containment, write $\mu_{\SS}$ for its M\"obius function. Then the fact that \[c(S)=\sum_{T\subseteq S,\ T\in\SS}a_{\SS}(T)\] tells us that we can write \[a_{\SS}(S)=\sum_{T\in\SS} c(T)\mu_{\SS}(T,S),\] which means that \[\ec_{\SS}(M)=\sum_{S,T\in\SS}c(T)(k-\rk S)\mu_{\SS}(T,S).\] (Note that this is the same as summing over only the pairs $S,T$ with $T\subseteq S$, since if $T\not\subseteq S$, $\mu(T,S)=0$.) From this perspective, it seems natural to define a version of $a_{\SS}$ which splits up the sum the other way, that is, we define \[b_{\SS}(T)=\sum_{S\in\SS}(k-\rk S)\mu_{\SS}(T,S),\] and from here we may clearly write \[\ec_{\SS}(M)=\sum_{T\in\SS}c(T)b_{\SS}(T).\]

A small advantage of singling out $b$ is that it clarifies the behavior of these operations under dualization:

\begin{prop}
If $\SS\subseteq\mathcal{P}(E)$ is some collection of sets, let $\SS'=\{E-S:S\in\SS\}$. Then:
\begin{enumerate}
\item $\ec_{\SS}(M)=\ec_{\SS'}(M^*)$
\item For $S\in\SS$, $a_{\SS}(S)=b_{\SS'}(E-S)$, where the latter is computed in $M^*$.
\end{enumerate}
\end{prop}
\begin{proof}
The rank of $E-S$ in $M^*$ is $\#(E-S)-k+\rk_MS$. So $c(E-S)$ in $M^*$ is $k-\rk_MS$, and $c(S)$ in $M$ is $k-\rk_{M^*}(E-S)$. So since \[\ec_{\SS}=\sum_{S,T\in\SS}c(T)(k-\rk S)\mu_{\SS}(T,S),\] (1) follows from the fact that $\SS'$ is the opposite poset to $\SS$, and therefore $\mu_{\SS'}(E-S,E-T)=\mu_{\SS}(T,S)$. From this perspective, (2) is also immediate.
\end{proof}

What is the point of going through this? Our immediate goal is to determine the conditions under which the expected codimension can be computed with respect to some set other than $\mathcal{P}(E)$ and still give the same answer. To figure this out, it would be enough to establish a condition for when $\ec_{\SS}(M)=\ec_{\SS-\{Z\}}(M)$ for some set $Z$. In fact, we can do a little better:

\begin{prop}
\label{prop:removechange}
If $\SS\subseteq\mathcal{P}(E)$ and $Z\in\SS$, then
\begin{enumerate}
\item $\ec_{\SS}(M)-\ec_{\SS-\{Z\}}(M)=a_{\SS}(Z)b_{\SS}(Z)$.
\item For $S\in\SS-\{Z\}$, $a_{\SS}(S)-a_{\SS-\{Z\}}(S)=a_{\SS}(Z)\mu_{\SS}(Z,S)$.
\item For $S\in\SS-\{Z\}$, $b_{\SS}(S)-b_{\SS-\{Z\}}(S)=\mu_{\SS}(S,Z)b_{\SS}(Z)$.
\end{enumerate}
\end{prop}
\begin{proof}
We have \[\ec_{\SS}(M)-\ec_{\SS-\{Z\}}(M)=\sum_{S,T\in\SS}c(T)(k-\rk S)(\mu_{\SS}(T,S)-\mu_{\SS-\{Z\}}(T,S))\]
if we take $\mu_{\SS-\{Z\}}(T,S)$ to be zero if either $T$ or $S$ is equal to $Z$. Recall that the M\"obius function can be defined by setting $\mu_{\SS}(X,Y)$ to be the sum over all chains in $\SS$ connecting $X$ to $Y$ of $(-1)^c$ where $c$ is the length of the chain. So $\mu_{\SS}(T,S)-\mu_{\SS-\{Z\}}(T,S)$ is going to be the alternating sum of lengths of chains in $\SS$ connecting $T$ to $S$ through $Z$; all other chains will appear in both sums and therefore cancel.

Write $q_k(X,Y)$ for the number of length-$k$ chains in $\SS$ connecting $X$ to $Y$. The number of length-$c$ chains connecting $T$ to $S$ through $Z$ is clearly equal to \[\sum_{k=0}^cq_k(T,Z)q_{c-k}(Z,S),\] which means that \[\mu_{\SS}(T,S)-\mu_{\SS-\{Z\}}(T,S)=\sum_c\sum_k(-1)^k(-1)^{c-k}q_k(T,Z)q_{c-k}(Z,S),\] which is just $\mu_{\SS}(T,Z)\mu_{\SS}(Z,S)$.

Therefore our difference of expected codimensions works out to be \[\sum_{S,T\in\SS}c(T)(k-\rk S)\mu_{\SS}(T,Z)\mu_{\SS}(Z,S)=a_{\SS}(Z)b_{\SS}(Z).\]

Dropping in this expression for the difference of M\"obius functions into the earlier expression of $a$, we see that \[a_{\SS}(S)-a_{\SS-\{Z\}}(S)=\sum_Tc(T)\mu_{\SS}(T,Z)\mu_{\SS}(Z,S)=a_{\SS}(Z)\mu_{\SS}(Z,S),\] and again similarly for $b$.
\end{proof}

\begin{cor}
\label{cor:removemnz}
Given $\AA\subseteq\SS$, if $a_{\SS}(A)=0$ for each $A\in\AA$, then $\ec_{\SS-\AA}(M)=\ec_{\SS}(M)$, and similarly with $a$ replaced with $b$.
\end{cor}
\begin{proof}
Remove the elements of $\AA$ from $\SS$ one by one. By part 1 of the proposition, removing something for which $a=0$ doesn't change $\ec$, and by part 2, the remaining elements of $\AA$ will still have $a=0$ after some have been removed. The argument is exactly analogous for $b$.
\end{proof}

This result will be a lot more useful if we can find a lot of sets for which $a$ and $b$ are zero. Luckily, we can:

\begin{prop}
\label{prop:disconnected}
Suppose that $S\in\SS$ is disconnected, say $S=\bigoplus_iS_i$. Suppose further that, for each $T\subseteq S$ for which $T\in\SS$, we also have that each connected component of $T$ is in $\SS$. Then $a_\SS(S)=0$.
\end{prop}
\begin{proof}
Recall that \[a_\SS(S)=c(S)-\sum_{T\subsetneq S}a_{\SS}(T).\] Any $T\subsetneq S$ which intersects more than one of the $S_i$'s is disconnected, so for those sets we may inductively conclude that $a_\SS(T)=0$. We are left only with sets that are contained in one of the $S_i$'s. To handle those, we note that $\sum_{T\subseteq S_i}a_{\SS}(T)=c(S_i)$. So we are left with $a_\SS(S)=c(S)-\sum_ic(S_i)$. It is simple to check that $c$ is additive in direct sums, so this zero.
\end{proof}

Simply by dualizing everything, we get a version of this statement about $b$. Suppose that $S\in\SS$ is such that $M/S$ is disconnected, and that whenever $T\supseteq S$ is in $\SS$, say $M/T=\bigoplus A_i$, we have each $T\cup A_i\in\SS$. Then $b_\SS(S)=0$.

In particular, Proposition \ref{prop:disconnected} and Corollary \ref{cor:removemnz} together imply that, starting with all of $\mathcal{P}(E)$, we can remove any number of disconnected sets, or any number of sets $S$ for which $M/S$ is disconnected, and end up with the same expected codimension, because the extra condition in Proposition \ref{prop:disconnected} will be trivially satisfied. Note that it doesn't say that we can remove sets of both kinds at the same time: Proposition \ref{prop:removechange} tells us that removing sets for which $b=0$ doesn't change values of $b$ for other sets, but values of $a$ can and will change.

First we need a lemma:

\begin{lem}
\label{lem:connectedness}
Suppose that $M$ is connected and that $S\subseteq M$ is connected. Say $M/S=\bigoplus_iA_i$ where each $A_i$ is connected in $M/S$. Then each $A_i\cup S$ is connected in $M$.
\end{lem}
\begin{proof}
Suppose $A_i\cup S$ is disconnected. Write $A=A_i\cup S$ and $B=\bigcup_{j\ne i}A_j\cup S$, so that $M/S=(A-S)\oplus(B-S)$. Because $S$ is connected, it must be contained in a connected component of $A$. Dually, since $A/S\cong A_i$ is connected, $S$ must contain all but one connected component of $A$. So in fact $S$ is a connected component of $A$, say $A=S\oplus C$.

We know that $\rk M-\rk S=(\rk A-\rk S)+(\rk B-\rk S)$, but our decomposition of $A$ gives us that the right-hand side is $\rk C+\rk B-\rk S$, so $\rk M=\rk B+\rk C$. So in fact $M=B\oplus C$, contradicting the connectedness of $M$.
\end{proof}

Note that by applying the theorem inductively to the $A_i$'s themselves, we get that any $S\cup\bigcup_{i\in I}A_i$ is also connected. Again we can extract a dual version of this statement: if $M$ and $M/S$ are connected but $S=\bigoplus B_i$ with each $B_i$ connected, then each $M/B_i$ is connected.

\begin{thm}
\label{thm:ecinvariance}
Suppose that $M$ is connected, that $\SS$ contains every set $S$ for which both $S$ and $M/S$ are connected, and that whenever $S\in\SS$, all of the connected components of $S$ are also in $\SS$. Then $\ec_{\SS}(M)=\ec(M)$.
\end{thm}
\begin{proof}
Starting with all of $\mathcal{P}(E)$, using Corollary \ref{cor:removemnz} we may remove every set $S$ for which $M/S$ is disconnected and $S\notin\SS$. Call the resulting collection $\TT$. If $T\in\TT-\SS$, we know that $M/T$ is connected, or we would have removed it already. So $T$ must be disconnected, or else it would be in $\SS$. Write $T=\bigoplus_i T_i$ with each $T_i$ connected.

We know that the $T_i$'s themselves are in $\TT$: each $M/T_i$ is connected by the dual version of Lemma \ref{lem:connectedness}, so in fact each $T_i\in\SS$. Suppose $U\subsetneq T$ and $U\in\TT$. If $U\in\TT-\SS$, then $a_\TT(U)=0$ by induction. Otherwise, $U\in\SS$, so its connected components are in $\SS$ by hypothesis, and we again can conclude inductively that $a_\TT(U)=0$. This is enough to be able to apply Proposition \ref{prop:disconnected} to get that $a_\TT(T)=0$.

So by applying Corollary \ref{cor:removemnz} once more, we may remove every set in $\TT-\SS$, which gives the result.
\end{proof}

Note that, in particular, Lemma \ref{lem:connectedness} implies that taking $\SS$ to be the collection of \emph{all} sets $S$ for which both $S$ and $M/S$ is connected will satisfy the hypotheses of Theorem \ref{thm:ecinvariance}. (These sets are called \defof{flacets}, and come up in the study of matroid polytopes. See \cite[2.6]{matpoly}.)

Expected codimension turns out to be well-behaved under direct sums:

\begin{prop}
\label{prop:ecdirectsum}
Let $M$ and $N$ be matroids on sets $E$ and $F$, and take collections $\SS\subseteq\mathcal{P}(E)$ and $\TT\subseteq\mathcal{P}(F)$. In $\mathcal{P}(E\cup F)$, let \[\AA=\SS\cup\TT\cup\{S\cup T:S\in\SS,\ T\in\TT\}.\] Then $\ec_\AA(M\oplus N)=\ec_\SS(M)+\ec_\TT(N)$.
\end{prop}
\begin{proof}
Take $A\in\AA$. If $A$ is a union of nonempty sets from $\SS$ and $\TT$, then it satisfies the hypotheses of Proposition \ref{prop:disconnected}. Otherwise, if $A\in\SS$, then $a_\AA(A)=a_\SS(A)$, and similarly for $\TT$. So in fact, \begin{align*}\ec_\AA(M\oplus N)&=\sum_{A\in\AA}a_\AA(A)\codim(A)\\&=\sum_{A\in\SS}a_\SS(A)\codim(A)+\sum_{A\in\TT}a_\TT(A)\codim(A)\\&=\ec_\SS(M)+\ec_\TT(N).\end{align*}
\end{proof}

In particular, using Proposition \ref{prop:dimdirectsum}, we see that if $M$ and $N$ have expected codimension, so does $M\oplus N$. Since it is trivial to check that both matroids on a one-element set have expected codimension, this also applies to loop and coloop extensions.

\begin{figure}[t]
\includegraphics[width=8cm]{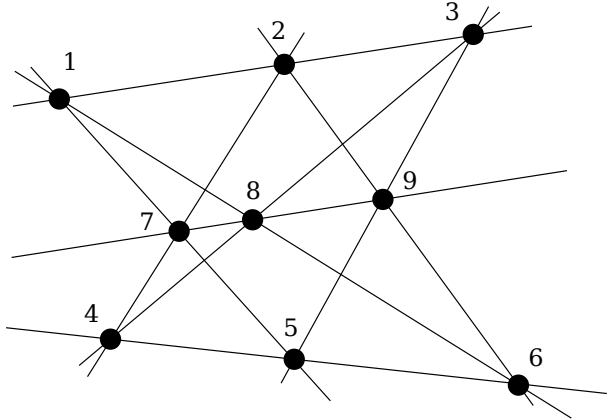}
\caption{A projective model of the ``Pappus matroid.''}
\label{fig:pappus}
\end{figure}

We conclude this section with an example of a matroid that doesn't have expected codimension:

\begin{cex}
\label{cex:pappus}
Consider the \defof{Pappus matroid} $P$, the rank-3 matroid on $[9]$ generated by the collinearities in Figure \ref{fig:pappus}. The only sets $S\subseteq[9]$ for which both $S$ and $P/S$ are connected are the nine sets of points which lie on lines in the picture. (That is, 123, 456, 789, 157, 168, 247, 269, 348, and 359.) From this we can easily compute that $\ec(P)=9$. However, the actual codimension of $X(P)$ in $G(3,9)$ is 8. This can be (and was) computed directly with a computer algebra system like Macaulay2; it also follows from computations performed in \cite{fnrmat}. Either way, $P$ doesn't have expected codimension.

This should not be especially surprising: the whole point of the Pappus matroid is that it demonstrates Pappus's theorem, that is, the fact that given any eight of the collinearities in Figure \ref{fig:pappus}, the ninth comes for free. Our definition of expected codimension is unable to keep track of ``global'' constraints like this one, so it treats all nine rank conditions as independent.
\end{cex}

\section{Positroid Varieties}
It is, as we have already observed, hopeless to expect to be able to say anything especially nice about the expected codimension of a general matroid variety, and we have already seen an example where our definition fails to produce the actual codimension of a matroid variety. There is, however, a much nicer class of matroids for which we will be able to say a lot more.

\begin{prop}
\label{prop:posdefs}
If $M$ is a matroid of rank $k$ on $[n]$, the following are equivalent:
\begin{enumerate}
\item $M$ is generated by rank conditions on cyclic intervals. (That is, cyclic permutations of intervals.)
\item $M$ is the matroid of a collection of vectors $v_1,\ldots,v_n$ in $\R^k$ for which all $k\times k$ minors are nonnegative.
\item $X(M)$ is the image of a Richardson variety in the flag variety $Fl(n)$ under the natural projection map $Fl(n)\to G(k,n)$.
\end{enumerate}
\end{prop}
\begin{proof}
For the equivalence of (1) and (2), see \cite{posschub}. For (2) and (3), see \cite{klspos}.
\end{proof}

\begin{defn}
\label{def:positroid}
A matroid satisfying any of the equivalent conditions just listed is called a \defof{positroid}, and its matroid variety is called a \defof{positroid variety}. (Positroids were first introduced and studied in \cite{totalpos}.)
\end{defn}

Note that, in particular, Schubert and Richardson matroids are positroids. Positroid varieties have many nice geometric properties \cite{projrich}. In particular, they are always reduced, irreducible, and Cohen-Macaulay, and unlike general matroid varieties (see Counterexample \ref{cex:pluckergen}) they are always cut out by Pl\"ucker variables. Positroids are very well-studied already, and there are several different combinatorial gadgets that have been invented to describe them, some of which are described in \cite{klspos}.

Because a positroid is generated by rank conditions on cyclic intervals, we can describe it completely by listing the rank of each cyclic interval.

\begin{defn}[\cite{klspos}]
\label{def:posrankmat}
Take a positroid $P$ on $[n]$. We'll think of elements $[n]$ as representatives of equivalence classes of integers mod $n$ with the obvious cyclic order (that is, 1 comes right after $n$), and we will use interval notation with this in mind; for example, if $n=6$, then we'll write $[5,8]=[5,2]=\{5,6,1,2\}$. In particular, $[5,5]=\{5\}$, whereas $[5,11]=[5,10]=[5,4]=\{5,6,1,2,3,4\}$. We can form a \defof{cyclic rank matrix} by setting $r_{ij}=\rk_P([i,j])$ for any integers $i,j$ with $0\le j-i\le n$.
\end{defn}

For any such matrix of numbers, the following conditions are necessary and sufficient for it to have arisen from this procedure:
\begin{itemize}
\item Each $r_{ii}$ is 0 or 1.
\item For any $i,j$, either $r_{i-1,j}=r_{ij}$ or $r_{i-1,j}=r_{ij}+1$, and similarly for $r_{i,j+1}$.
\item If $r_{ij}=r_{i-1,j}=r_{i,j+1}$, then $r_{i-1,j+1}=r_{ij}$.
\end{itemize}

\begin{defn}[\cite{klspos}]
\label{def:posaffine}
Given a cyclic rank matrix corresponding to a positroid on $[n]$, we can form an \defof{affine permutation}. This will be a bijection $\pi:\Z\to\Z$ such that $i\le\pi(i)\le i+n$ and $\pi(i+n)=\pi(i)+n$ for all $i$. We define $\pi$ as a matrix by putting a 1 in position $(i,j)$ if $r_{ij}=r_{i,j-1}=r_{i+1,j}\ne r_{i+1,j-1}$ and putting a 0 there otherwise. One can check that each row and each column will have exactly one 1. Note that to describe $\pi$ it's enough to describe the images of the elements of $[n]$.

One can reverse this process to read off the rank matrix from the affine permutation matrix. Given an interval $[i,j]$, consider the entries of the matrix weakly southwest of position $(i,j)$ in the affine permutation matrix. (That is, positions $(k,l)$ for which $[k,l]\subseteq[i,j]$.) The rank of $[i,j]$ then works out to be $\#[i,j]-d$, where $d$ is the number of 1's among the these entries. Equivalently, thinking of the affine permutation as a function, we can say \[d=\#\{k:i,k,\pi(k),\pi(i)\mbox{ occur cyclically consecutively}\}.\]

It's also possible to determine the codimension of a positroid variety from an affine permutation. For an affine permutation $\pi$, define the \defof{length} of $\pi$, written $l(\pi)$, as the number of \defof{inversions}, that is, the number of pairs $i,k$ with $i,k,\pi(k),\pi(i)$ occurring cyclically consecutively. Each of these pairs will correspond to a pair of 1's in the affine permutation matrix arranged southwest-to-northeast. Then $l(\pi)$ is the codimension of the positroid variety corresponding to $\pi$. (This is proved in \cite[5.9]{klspos}.)
\end{defn}

The affine permutation gives a simple way to determine which rank conditions are necessary to define $X(P)$:

\begin{defn}
\label{def:posessential}
The \defof{essential set} of an affine permutation is defined by the following procedure: cross out all the positions strictly below or to the left of a 1 in the affine permutation matrix, and take the positions which are at the upper-right corners of their connected components. (This definition follows Fulton's description in \cite{fultonessential}, though he did not refer to positroid varieties.)
\end{defn}

By convention, we don't take positions on the upper-right edge of the matrix (that is, ones where $j-i=n$) to be essential. Imposing the rank conditions corresponding to the essential intervals are enough to define a positroid variety in $G(k,n)$ as a scheme.

\begin{exmp}
\label{ex:positroid}
The positroid of rank 3 on $[6]$ generated by forcing $[1,3]$, $[3,5]$, and $[5,1]$ to have rank 2 has the following cyclic rank matrix:

\[
\begin{array}{cccccccccccc}
1&2&\underline{\mathbf 2}&3&3&3&3& & & & &\\
 &1&2&3&3&\underline{\mathbf 3}&3&3& & & &\\
 & &1&2&\underline{\mathbf 2}&3&3&3&3& & &\\
 & & &1&2&3&3&\underline{\mathbf 3}&3&3& &\\
 & & & &1&2&\underline{\mathbf 2}&3&3&3&3&\\
 & & & & &1&2&3&3&\underline{\mathbf 3}&3&3\\
\end{array}
\]

We think of the matrix as repeating infinitely in the northwest and southeast directions. So, for example, the 3 printed in the fourth row and fourth nonempty column indicates that $[4,7]=[4,1]=\{4,5,6,1\}$ has rank 3. The underlined entries are the positions of the 1's in the corresponding affine permutation matrix. We'll sometimes write affine permutations as functions, listing the image of each element of $[n]$ in order. So, for example, this one is $3,6,5,8,7,10$.
\end{exmp}

We have aleady seen in Counterexample \ref{cex:pappus} a case in which the expected codimenion of a matroid variety fails to line up with its actual codimension in the Grassmannian. The main result of this section is that that doesn't happen for positroids:

\begin{thm}
\label{thm:posexpcodim}
Positroids have expected codimension.
\end{thm}

In order to prove this, we're going to need to understand a little bit more about the matroid structure of a positroid. We refer repeatedly to restrictions and contractions of positroids; note that it follows directly from the second definition in Proposition \ref{prop:posdefs} that these are again positroids.

\begin{lem}
\label{lem:connectedintervals}
If $P$ is a positroid on $[n]$, $X\subseteq [n]$, and both $P|_X$ and $P/X$ are connected, then $X$ is an interval.
\end{lem}
\begin{proof}
Suppose $X$ is not an interval. Take $c_1,c_2\in[n]-X$ to lie in two different cyclic intervals of $[n]-X$. Since $P/X$ is connected, there is a circuit $C$ of $P/X$ which contains both $c_1$ and $c_2$. By restricting to $X\cup C$, we may assume that $P/X$ is a circuit. Similarly, for $b_1,b_2\in X$ lying on different sides of $c_1$ and $c_2$ (so the named elements appear in the cyclic order $b_1,c_1,b_2,c_2$), there is a circuit $B$ of $(P|_X)^*=P/([n]-X)$ containing both, and we may contract the elements of $X-B$ and assume that $(P|_X)^*$ is a circuit, that is, that everything in $X$ is parallel.

Now, delete all elements of $X$ other than $b_1$ and $b_2$. This doesn't change the rank of any set in $P/X$: everything in $X$ was parallel to $b_1$ and $b_2$. Dually, contract all elements of $[n]-X$ except $c_1$ and $c_2$. Now we have $n=4$, and the sets $\{1,3\}$ and $\{2,4\}$ each have rank 1. This matroid is not a positroid, which can easily be checked, so we have a contradiction.
\end{proof}

\begin{lem}
\label{lem:noncrossing}
The connected components of a positroid form a non-crossing partition. (This was also proved independently in \cite{noncross}.)
\end{lem}
\begin{proof}
Suppose first that $P$ has just two connected components, say $P=A\oplus B$. Then $P/A=B$ and $P/B=A$ are also both connected, so Lemma \ref{lem:connectedintervals} implies that they are both intervals. If there are more than two connected components, they no longer have to both be intervals, but for any two components $C$ and $D$, each of $C$ and $D$ must be an interval inside $C\cup D$, which means in particular that they cannot cross.
\end{proof}

\begin{lem}
\label{lem:intervalcomponents}
If $P$ is a connected positroid on $[n]$ and $I\subseteq [n]$ is an interval, then each connected component of $I$ is an interval.
\end{lem}
\begin{proof}
Say $I=X\oplus\bigoplus_iY_i$ with $X$ and each $Y_i$ connected, and suppose $X$ is not an interval, say $X=\bigcup_kJ_k$ and $I-X=\bigcup_lJ'_l$ where each $J_k$ and $J'_l$ is an interval. Since the components of $I$ have to form a non-crossing partition by Lemma \ref{lem:noncrossing}, none of the $Y_i$'s can meet more than one of the $J'_l$'s. So we may assume that left and right endpoints of $X$ coincide with those of $I$ by removing the $Y_i$'s that lie to the left of $X$'s left endpoint or to the right of its right endpoint. We now know that all the $J'_l$'s lie in between two $J_k$'s.

Just as in the proof of Lemma \ref{lem:connectedintervals}, the connectedness of $X$ lets us conclude that there is a circuit of $X^*=P/([n]-X)$ that contains points in two different $J_k$'s. Suppose there is a circuit of $P/X$ that contains a point of $I-X$ and a point of $P-I$. If this were the case, because we forced all the points of $I-X$ to lie between intervals of $X$, we would be in the exact situation that gave us a contradiction in the previous proof.

So there must be no such circuit. But this means that \[P/X=I/X\oplus((P-I)\cup X)/X=\left(\bigoplus_iY_i\right)\oplus(P-I\cup X/X),\] which means that \[\rk P-\rk X=\sum_i\rk Y_i+\rk(P-I\cup X)-\rk X,\] so in fact \[P=\left(\bigoplus_iY_i\right)\oplus((P-I)\cup X),\] contradicting the connectedness of $P$.
\end{proof}

\begin{lem}
\label{lem:npositroid}
For a positroid $P$, let $\II\subseteq\mathcal{P}([n])$ be the collection of all cyclic intervals. For any interval $[i,j]\ne[n]$, $a_\II([i,j])$ is equal to the entry (either 0 or 1) at $(i,j)$ in the affine permutation matrix.
\end{lem}
\begin{proof}
To see this, it's enough to note that our purported $a_\II$ satisfies the relation \[c(S)=\sum_{T\subseteq S}a_\II(T).\] We mentioned above that $c([i,j])=\#[i,j]-\rk[i,j]$ is the number of intervals $[k,l]\subseteq[i,j]$ with a 1 in the affine permutation matrix at position $(k,l)$, so this is true.
\end{proof}

\begin{proof}[Proof of Theorem \ref{thm:posexpcodim}]
First, we claim that for a positroid $P$, if $\II\subseteq\mathcal{P}([n])$ is the collection of all cyclic intervals of $[n]$, $\ec(P)=\ec_\II(P)$. For connected positroids, this follows immediately from Theorem \ref{thm:ecinvariance}: Lemmas \ref{lem:connectedintervals} and \ref{lem:intervalcomponents} say that taking $\SS=\II$ satisfies the hypotheses of the theorem. For a general positroid, we can decompose it as a direct sum and apply Proposition \ref{prop:ecdirectsum}.

So now it remains to show that $\ec_\II(P)$ is the actual codimension of $P$. We'll use the fact that $\codim P=l(\pi)$, where $\pi$ is the corresponding affine permutation. Let $d_I$ be the number of intervals $[k,l]\subseteq I$ with a 1 in position $(k,l)$, and let $\II'$ be the collection of intervals other than $[n]$ with a 1 at the corresponding position. We can ignore $[n]$ itself because $\codim[n]=0$, so we may compute, using Lemma \ref{lem:npositroid}:
\[\ec_\II(P)=\sum_{I\in\II}a_\II(I)\codim(I)=\sum_{I\in\II'}\codim(I)=\sum_{I\in\II'}(k-(\#I-d_I)),\] where $k$ is the rank of $P$. A simple computation (which is spelled out in \cite{klspos}) shows that $\sum\#I=nk+n$, and we know that $\sum d_I=l(\pi)+n$, since we're counting the pairs of intervals in the definition of $l(\pi)$ and also counting the pairs $(I,I)$. So we're left with \[\ec_\II(P)=nk-(nk+n)+l(\pi)+n=l(\pi).\]
\end{proof}

\section{Valuativity}
\label{sec:valuativity}
Let $M$ be a matroid on a set $[n]$. For each basis $B$ of $M$, consider the vectors in $\R^n$ whose entries are 1 if the corresponding element of $[n]$ is in $B$ and 0 otherwise. The convex hull of these vectors is called the \defof{matroid polytope} of $M$, written $P(M)$. There are many examples of combinatorial properties of matroids that are encoded in the geometry of the matroid polytope.

\begin{defn}
Given a matroid $M$, a \defof{matroidal subdivision} of $P(M)$ is a decomposition of $P(M)$ into polytopes which are all matroid polytopes. If $\mathcal{D}$ is a matroidal subdivision of $P(M)$, write $\mathcal{D}_{in}$ for the internal faces of $\mathcal{D}$, that is, the faces of $\mathcal{D}$ that are not also faces of $P(M)$. Given a function $f$ from the set $\mathrm{Mat}(n)$ of all matroids on $[n]$ into an abelian group, we say $f$ is \defof{valuative} if, for any matroid $M$ and any matroidal subdivision $\mathcal{D}$ of $P(M)$, \[f(M)=\sum_{P(N)\in\mathcal{D}_{in}}(-1)^{\dim P(M)-\dim P(N)}f(N).\]
\end{defn}

Valuative matroid invariants are studied in detail in \cite{valuative}. We single out the following result, which appears as \cite[5.4]{valuative} in slightly different language:

\begin{thm}
\label{thm:valschub}
The set of Schubert matroids forms a basis for $\mathrm{Mat}(n)$ modulo matroidal subdivisions.
\end{thm}

We will show:

\begin{thm}
\label{thm:expval}
Expected codimension is a valuative matroid invariant.
\end{thm}

Since Schubert matroids, being positroids, have expected codimension, Theorem \ref{thm:valschub} gives us another way to think about expected codimension: you could have defined it by assigning Schubert matroids their codimensions and extending to all matroids by subdividing the matroid polytope and insisting that it be valuative.

We'll prove Theorem \ref{thm:expval} by proving something stronger first:

\begin{lem}
\label{lem:tuttelike}
Let $M$ be a matroid on $[n]$, and define \[s_M(x,y,z)=\sum_{S\subseteq T\subseteq [n]}x^{\#S-\rk S}y^{\rk M-\rk T}z^{\#T-\#S}.\] Then the function that takes $M$ to $s_M$ is valuative.
\end{lem}

This is a generalization of the \defof{Tutte polynomial}, which is \[t_M(x,y)=s_M(x-1,y-1,0).\] In \cite[6.4]{troplin}, David Speyer shows that the Tutte polynomial is a valuative matroid invariant. The proof we give here of \ref{lem:tuttelike} turns out to be almost identical. We single out the following lemma, which appears as \cite[6.5]{troplin}:

\begin{lem}
\label{lem:euler}
If $P$ is a polytope and $\mathcal{G}$ is the set of internal faces of a decomposition of $P$, then \[\sum_{N\in\mathcal{G}}(-1)^{\dim P-\dim N}=1.\]
\end{lem}
\begin{proof}
This is just $(-1)^{\dim P}(\chi(P)-\chi(\partial P))$ where $\chi$ is the Euler characteristic. So the result follows, becuase $P$ is contractible and $\partial P$ is homeomorphic to a $(\dim P-1)$-sphere.
\end{proof}

\begin{proof}[Proof of Lemma \ref{lem:tuttelike}]
Plugging the definition of $s_M$ to the definition of valuativity, it's enough to show, for any matroidal subdivision $\mathcal{D}$ of $P(M)$, \[x^{\#S-\rk_M S}y^{\rk M-\rk_M T}z^{\#T-\#S}\] is equal to \[\sum_{P(N)\in\mathcal{D}_{in}}(-1)^{\dim P(M)-\dim P(N)}x^{\#S-\rk_N S}y^{\rk N-\rk_N T}z^{\#T-\#S}.\] Comparing the coefficients of $x^{\#S-\rk_M S}y^{\rk M-\rk_M T}z^{\#T-\#S}$, we want that \[\sum_{\substack{P(N)\in\mathcal{D}_{in}\\\rk_N(S)=r,\ \rk_N(T)=s}}(-1)^{\dim P(M)-\dim P(N)}\] is 1 if $(r,s)=(\rk_M(S),\rk_M(T))$ and 0 otherwise. Since the sum is empty if $r>\rk_M(S)$ or $s>\rk_M(T)$, we'll just show that \[\sum_{\substack{P(N)\in\mathcal{D}_{in}\\\rk_N(S)\ge r,\ \rk_N(T)\ge s}}(-1)^{\dim P(M)-\dim P(N)}=1\] for $r\le \rk_M(S)$ and $s\le \rk_M(T)$.

Let $l_S$ be the linear function on $\R^n$ sending $(x_i)$ to $\sum_{i\in S}x_i$. Note that \[\rk_N(S)=\max_{x\in P(N)}l_S(x).\] So $\rk_N(S)\ge r$ if and only if $P(N)$ intersects the half-space $l_S>r-\frac12$. So our equality follows by applying Lemma \ref{lem:euler} to $P(M)\cap\{l_M>r-\frac12\}\cap\{l_N>s-\frac12\}$.
\end{proof}

\begin{proof}[Proof of Theorem \ref{thm:expval}]
This follows directly: \[\ec(M)=\frac\partial{\partial x}\frac\partial{\partial y}s_M(0,0,-1).\]
\end{proof}

\bibliography{ecodimbib}
\bibliographystyle{alpha}

\end{document}